\documentclass[11pt,english]{article}
\usepackage[T1]{fontenc}
\usepackage[latin9]{inputenc}
\usepackage{amsmath}
\usepackage{amssymb}

\makeatletter
\usepackage{amsfonts}

\providecommand{\U}[1]{\protect\rule{.1in}{.1in}}
\newtheorem{theorem}{Theorem}

\newtheorem{corollary}[theorem]{Corollary}

\newtheorem{lemma}[theorem]{Lemma}

\newtheorem{proposition}[theorem]{Proposition}

\newenvironment{proof}[1][Proof]{\noindent\textbf{#1.} }{\ \rule{0.5em}{0.5em}}

\makeatother

\usepackage{babel}

\begin{document}

\title{Noncommutative Ergodic Theorems}

\author{Anders Karlsson%
\thanks{Royal Swedish Academy of Sciences Research Fellow supported by a grant
from the Knut and Alice Wallenberg Foundation. Supported also by the
Swedish Research Council.%
}\\
Department of Mathematics \\
Royal Institute of Technology\\
100 44 Stockholm \\
Sweden \and François Ledrappier%
\thanks{Supported in part by NSF grant DMS-0500630.%
}\\
Department of Mathematics \\
University of Notre Dame\\
Notre Dame, IN 46556\\
U.S.A. }

\maketitle
\centerline{\emph{To Robert J. Zimmer on the occasion of his sixtieth
birthday}}
\begin{abstract}
We present recent results about the asymptotic behavior of ergodic
products of isometries of a metric space $X$. If we assume that the
displacement is integrable, then either there is a sublinear diffusion
or there is, for almost every trajectory in $X$, a preferred direction
at the boundary. We discuss the precise statement when $X$ is a proper
metric space (\cite{KL1}) and compare it with classical ergodic theorems.
Applications are given to ergodic theorems for nonintegrable functions,
random walks on groups and Brownian motion on covering manifolds.
\end{abstract}
\

In this note, we survey some recent results about the asymptotic behavior
of ergodic products of 1-Lipschitz mappings of a metric space $(X,d)$.
If the mappings are translations on the real line $(\mathbb{R},|\cdot|)$,
then classical ergodic theorems apply, as we recall in Section 1.
In more general settings, a suitable generalization of the convergence
of averages is the ray approximation property: a typical orbit stay
within a $O(\frac{1}{n})$ distance of some (random) geodesic ray
(\cite{Pa}, \cite{K3} and \cite{KM}, see Theorem \ref{KM} below).
Most of this note is devoted to another generalization, valid in the
case when the space $(X.d)$ is proper (see Theorem \ref{main}).
It also says that there is a (random) direction followed by the typical
trajectory, but now a direction is just a point in the metric compactification
of $(X,d)$. We discuss in Section 3 how Theorem \ref{main} yields
the ray approximation property when the space $(X,d)$ is a CAT(0)
metric space, and consequently Oseledets Theorem (following \cite{K3}).
We give in Section 4 some applications when the space $(X,d)$ is
a Gromov hyperbolic space. In particular, by choosing different metrics
on $\mathbb{R}$ we directly show some known ergodic theorems for
nonintegrable functions. We prove Theorem \ref{main} in Section 5
and give applications to Random Walks in Section 6. Section 6 comes
from \cite{KL2}, with slightly simpler proofs. The gist of our results
is that for a random walk with first moment on a locally compact group
with a proper metric, the Liouville property implies that the linear
drift of the random walk, if any, completely comes from a character
on the group (see Section 6 for precise statements). This is to be
compared with the results of Guivarc'h (\cite{G}) in the case of
connected Lie groups. Since our result applies to discrete groups,
it can, through discretization, be applied to Brownian motion on Riemannian
covers of finite volume manifolds. We state in Section 7 the subsequent
result from \cite{KL3}.

\section{Classical Ergodic Theorems.}

We consider a Lebesgue probability space $(\Omega,\mathcal{A},\mathbb{P})$,
an invertible bimeasurable transformation $T$ of the space $(\Omega,\mathcal{A})$
that preserves the probability $\mathbb{P}$, a function $f:\Omega\to\mathbb{R}$,
and we define $S_{n}(\omega):=\sum_{i=0}^{n-1}f(T^{i}\omega).$ This
setting occurs in particular in Statistical Mechanics and in Mechanics,
where $\Omega$ is the space of configurations, $T$ the time 1 evolution
and $\mathbb{P}$ is either the statistical distribution of states
or the Liouville measure on the energy levels. The Ergodic Hypothesis
led to assert that the ergodic averages \[
\frac{1}{n}S_{n}(\omega)=\frac{1}{n}\sum_{i=0}^{n-1}f(T^{i}\omega)\]
 have some asymptotic regularity.

Around 1930, Koopman suggested that it might be useful to consider
the operator $U$ on functions $f$ in $L^{2}(\Omega,\mathbb{P})$
defined by\[
\left(Uf\right)(\omega)=f(T\omega).\]
 Since $T$ is measure-preserving, the operator $U$ is unitary. The
ergodic average then becomes\[
\frac{1}{n}\sum_{k=0}^{n-1}U^{k}f.\]
 The system $(\Omega,\mathcal{A},\mathbb{P};T)$ is said to be ergodic
if the only functions in $L^{2}$ which are invariant under the unitary
operator $U$ are the constant functions. In this text, for the sake
of exposition, we assume that the system $(\Omega,\mathcal{A},\mathbb{P};T)$
is ergodic. Statements for nonergodic systems follow using the decomposition
of the measure $\mathbb{P}$ into ergodic components. As an application
of the Spectral Theorem, von Neumann indeed proved:

\begin{theorem}\label{vonNeumann} {[}von Neumann Ergodic Theorem,
1931{]} Assume that the transformation $T$ is ergodic and that $\int f^{2}d\mathbb{P}<\infty$,
then\[
\frac{1}{n}\sum_{k=0}^{n-1}U^{k}f\rightarrow\int_{\Omega}fd\mathbb{P}\]
 in $L^{2}$. \end{theorem}

This prompted Birkhoff to prove an almost everywhere convergence theorem:
\begin{theorem}\label{Birkhoff}{[}Birkhoff Ergodic Theorem, 1931{]}
Assume that the transformation $T$ is ergodic, and that $\int\max(f,0)d\mathbb{P}<\infty$,
then for $\mathbb{P}$-almost every $\omega$, as $n\to\infty$: \[
\frac{1}{n}S_{n}(\omega)\;\rightarrow\;\int fd\mathbb{P}.\]
 \end{theorem}

A variant of the ergodic theorem applies to \textit{subadditive}  sequences.
A sequence $S_{n}$ of real functions on $\Omega$ is said to be subadditive
if, for $\mathbb{P}$-almost every $\omega$, all natural integers
$n,m$: \[
S_{n+m}(\omega)\;\leq\; S_{m}(\omega)+S_{n}(T^{m}\omega).\]
 \begin{theorem}\label{Kingman}{[}Kingman Subadditive Ergodic Theorem,
1968{]} Assume that the transformation $T$ is ergodic, and that $\int\max(S_{1},0)d\mathbb{P}<\infty$,
then for $\mathbb{P}$-almost every $\omega$, as $n\to\infty$: \[
\frac{1}{n}S_{n}(\omega)\;\rightarrow\;\inf_{n}\frac{1}{n}\int S_{n}d\mathbb{P}.\]
 \end{theorem}

Proofs of Theorems \ref{Birkhoff} and \ref{Kingman} often appeal
to some combinatorics of the sequence $S_{n}(\omega)$ along individual
orbits. The following technical Lemma was proven by the first author
and Margulis:

\begin{lemma}\label{key}{[}\cite{KM}, Proposition 4.2{]} Let $S_{n}$
be a subadditive sequence on an ergodic dynamical system $(\Omega,\mathcal{A},\mathbb{P};T)$.
Assume that $\int\max(S_{1},0)d\mathbb{P}<\infty$ and that $\alpha:=\inf_{n}\frac{1}{n}\int S_{n}d\mathbb{P}>-\infty$.
Then, for $\mathbb{P}$ a.e. $\omega$, all $\varepsilon>0$, there
exist $K=K(\omega)$ and an infinite number of instants $n$ such
that: \[
S_{n}(\omega)-S_{n-k}(T^{k}\omega)\;\geq\;(\alpha-\varepsilon)k\textrm{ for all }k,K\leq k\leq n.\]
 \end{lemma}

In particular, it follows from subadditivity that $\liminf_{k}\frac{S_{k}(\omega)}{k}\geq\alpha$.
Therefore Theorem \ref{Birkhoff} follows (in the case $\int fd\mathbb{P}>-\infty$)
because in that case, both sequences $S_{k}$ and $-S_{k}$ are subadditive
and \[
\inf_{n}\frac{1}{n}\int S_{n}d\mathbb{P}=\sup_{n}\frac{1}{n}\int S_{n}d\mathbb{P}=\int fd\mathbb{P}=\alpha.\]
 On the other hand, $\limsup_{k}\frac{S_{k}}{k}$ is a $T$-invariant
function which, by subadditivity, is not bigger than $\limsup_{p}\frac{1}{pk}\sum_{j=0}^{p-1}S_{k}(T^{jk}\omega)$.
Thus the constant $\limsup_{k}\frac{S_{k}}{k}$ is not bigger than
$\frac{1}{k}\int S_{k}d\mathbb{P}.$ Theorem \ref{Kingman} follows
in the case when $\alpha>-\infty.$ To treat the case $\alpha=-\infty$
in both theorems, it suffices to replace $S_{n}$ by $\max(S_{n},-nM)$,
and to let $M$ go to infinity, see \cite{Kr} for details.

\

\section{Noncommutative Ergodic Theorems.}

Observe that Theorem \ref{vonNeumann} also holds true for any linear
operator $U$ of a Hilbert space assuming $\left\Vert U\right\Vert \leq1$.
One can take one step further and define for any $g\in\mathcal{H}$,
$\phi(g):=Ug+f.$ Then $\phi$ is an isometry (or merely 1-Lipschitz
in the case $\left\Vert U\right\Vert \leq1$).

Note that\[
\phi^{n}(0)=\sum_{k=0}^{n-1}U^{k}f.\]
 Pazy proved in \cite{Pa} that more generally for any map $\phi:\mathcal{H}\rightarrow\mathcal{H}$
such that $\left\Vert \phi(x)-\phi(y)\right\Vert \leq\left\Vert x-y\right\Vert ,$
it holds that there is a vector $v\in\mathcal{H}$ such that\[
\frac{1}{n}\phi^{n}(0)\rightarrow v\]
 in norm. This can be reformulated as follows: There is a unit speed
geodesic $\gamma(t)=tv/\left\Vert v\right\Vert $ in $\mathcal{H}$
such that\begin{equation}
\frac{1}{n}\left\Vert \phi^{n}(0)-\gamma(n\left\Vert v\right\Vert )\right\Vert =\frac{1}{n}\left\Vert \phi^{n}(0)-nv\right\Vert \rightarrow0.\label{rayapprox}\end{equation}

We call this property { \textit{ray approximation}. It turns out
that this generalization of the ergodic theorem still holds for more
general group actions than the actions of $\mathbb{Z}$. Let $G$
be a second countable, locally compact, Hausdorff topological semi-group,
and consider $g:\Omega\to G$ a measurable map. We form \[
Z_{n}(\omega):=g(\omega)g(T\omega)...g(T^{n-1}\omega)\]
 and we ask whether $Z_{n}$ converges to infinity with some linear
speed.}

Assume $G$ acts on a metric space $(X,d)$ by 1-Lipschitz transformations.
Then for a fixed $x_{0}\in X$, we can define, for $g\in G$, $|g|:=d(x_{0},gx_{0}).$
Clearly, up to a bounded error, $|Z_{n}(\omega)|$ does not depend
on our choice of $x_{0}$. We have: \begin{proposition}\label{drift}
Assume the transformation $T$ is ergodic, and $\int|g|d\mathbb{P}<\infty$.
Then there is a nonnegative number $\alpha$ such that for $\mathbb{P}$-almost
every $\omega$, as $n\to\infty$: \[
\frac{1}{n}|Z_{n}(\omega)|\;\rightarrow\;\alpha.\]
 The number $\alpha$ is given by \begin{equation}
\alpha=\inf_{n}\frac{1}{n}\int|Z_{n}(\omega)|d\mathbb{P}.\label{alpha}\end{equation}
 \end{proposition} \begin{proof} It suffices to observe that the
sequence $|Z_{n}(\omega)|$ satisfies the hypotheses of Theorem \ref{Kingman}.
Our hypothesis says that $\int Z_{1}<\infty$. The subaditivity follows
from the 1-Lipschitz property: \begin{align*}
|Z_{n+m}(\omega)|\; & =\; d(x_{0},g(\omega)...g(T^{n+m-1}\omega)x_{0})\\
 & \leq\; d(x_{0},g(\omega)...g(T^{m-1}\omega)x_{0})+\\
 & \;+d(g(\omega)...g(T^{m-1}\omega)x_{0},g(\omega)...g(T^{n+m-1}\omega)x_{0})\\
 & \leq\;|Z_{m}(\omega)|+d(x_{0},g(T^{m}\omega)...g(T^{n+m-1}\omega)x_{0})\\
 & =|Z_{m}(\omega)|+|Z_{n}(T^{m}\omega)|.\end{align*}
 Moreover we see that the limit $\alpha$ is given by $\inf_{n}\frac{1}{n}\int|Z_{n}(\omega)|d\mathbb{P}.$
\end{proof}

\
 When $\alpha>0$, Proposition \ref{drift} says that the points $Z_{n}(\omega)x$
go to infinity with a definite linear speed. The question arises of
the convergence in direction of the points $Z_{n}(\omega)x$. Given
equation (\ref{rayapprox}), we expect that an almost everywhere convergence
theorem will say that $Z_{n}(\omega)x$ will stay at a sublinear distance
of a geodesic. We present several results in that direction depending
on different geometric hypotheses on the space $X$.

Assume $X$ is a complete, Busemann nonpositively curved and uniformly
convex (e.g. $CAT(0)$ or uniformly convex Banach space) metric space.
Then,

\begin{theorem}\label{KM} \cite{KM} Under these assumptions, there
is a constant $\alpha\geq0$ and, for $\mathbb{P}$-almost every $\omega$,
a geodesic ray $\gamma_{\omega}$ such that\[
\frac{1}{n}d(Z_{n}(\omega)x_{0},\gamma_{\omega}(n\alpha))\rightarrow0\text{.}\]

\end{theorem}

We outline the proof (see \cite{KM}, Section 5, for details). Let
$a(n,\omega)=d(x_{0},Z_{n}(\omega)x_{0})$ for each $n$. Consider
a triangle consisting of $x_{0},$ $Z_{n}(\omega)x_{0}$, and $Z_{k}(\omega)x_{0}.$
Note that the side of this triangle have lengths $a(n,\omega),$ $a(k,\omega),$
and (at most) $a(n-k,T^{k}\omega).$ Given $\varepsilon>0$ (and a.e.
$\omega)$, for $k$ large it holds that $a(k,\omega)\leq(\alpha+\varepsilon)k$.
Assume now in addition to $k$ being large that $n$ and $k$ are
as in Lemma \ref{key}. This implies that the triangle is thin in
the sense that $Z_{k}(\omega)x_{0}$ lies close to the geodesic segment
$[x_{0},Z_{n}(\omega)x_{0}],$ more precisely, the distance is at
most $\delta(\varepsilon)a(k,\omega),$ where $\delta$ only depends
on the geometry. Thanks to the geometric assumptions this $\delta(\varepsilon)$
tends to $0$ as $\varepsilon$ tends to $0$. Selecting $\varepsilon$
tending to $0$ fast enough we can by selecting suitable $n$ as in
Lemma \ref{key} and some simple geometric arguments obtain a limiting
geodesic. Finally, one has essentially from the contruction that as
$m\rightarrow\infty,$ the points $Z_{m}(\omega)x_{0}$ lie at a sublinear
distance from this geodesic ray$.$

\

This note is devoted to the generalization of the ergodic theorem
to groups of isometries of a metric space $(X,d)$. We assume that
the space $(X,d)$ is \textit{proper} (closed bounded subsets are
compact) and we consider the \textit{metric compactification} of $X$.
Define, for $x\in X$ the function $\Phi_{x}(z)$ on $X$ by: \[
\Phi_{x}(z)\;=\; d(x,z)-d(x,x_{0}).\]
 The assignment $x\mapsto\Phi_{x}$ is continuous, injective and takes
values in a relatively compact set of functions for the topology of
uniform convergence on compact subsets of $X$. The \textit{metric
compactification}  $\overline{X}$ of $X$ is the closure of $X$
for that topology. The \textit{metric boundary}  $\partial X:=\overline{X}\setminus X$
is made of Lipschitz continuous functions $h$ on $X$ such that $h(x_{0})=0$.
Elements of $\partial X$ are called \textit{horofunctions}. Our main
result is the following \begin{theorem}\label{main}{[}Ergodic Theorem
for isometries \cite{KL1}{]} Let $T$ be a measure preserving transformation
of the Lebesgue probability space $(\Omega,\mathcal{A},\mathbb{P})$,
$G$ a locally compact group acting by isometries on a proper space
$X$ and $g:\Omega\to G$ a measurable map satisfying $\int|g(\omega)|d\mathbb{P}(\omega)<\infty.$
Then, for $\mathbb{P}$-almost every $\omega$, there is some $h_{\omega}\in\partial X$
such that: \[
\lim_{n\to\infty}-\frac{1}{n}h_{\omega}(Z_{n}(\omega)x_{0})\;=\;\lim_{n\to\infty}\frac{1}{n}d(x_{0},Z_{n}(\omega)x_{0}).\]
 \end{theorem}

\
 For the convenience of the reader, the proof of Theorem \ref{main}
is given in Section 5. We explain in Section 3 why the convergence
in Theorem \ref{main} is equivalent to the ray approximation under
the $CAT(0)$ assumption. Note that by Theorem \ref{main} the former
convergence holds for all norms on $\mathbb{R}^{d}$, but that Theorem
\ref{main} does not apply to infinite dimensional Banach spaces.
In this case, one can use Lemma \ref{key} to prove a noncommutative
ergodic theorem with linear functionals of norm $1,$ somewhat analogous
to horofunctions. Namely,

\begin{theorem}\label{Ka} \cite{Ka} Let $Z_{n}(\omega)$ be an
ergodic integrable cocycle of $1$-Lipschitz self-maps of a reflexive
Banach space. Then for $\mathbb{P}$-almost every $\omega$ there
is a linear functional $f_{\omega}$ of norm 1 such that\[
\lim_{n\rightarrow\infty}\frac{1}{n}f_{\omega}(Z_{n}(\omega)0)=\alpha.\]

\end{theorem}

On the other hand, Kohlberg-Neyman \cite{KN} found a counterexample
to the norm convergence, or more precisely to (\ref{rayapprox}),
for general Banach spaces.

\

\section{Case when $X$ is a $CAT(0)$ proper space.}

When the space $(X,d)$ is a proper $CAT(0)$ metric space, both Theorems
\ref{KM} and \ref{main} apply. Because it is a direct generalization
of the important case when $G$ is a linear group, it is often called
the Oseledets Theorem. In this section we explain how to recover the
ray approximation and other more familiar forms of Oseledets Theorem
from Theorem \ref{main}. Many of the geometric ideas in this section
go back to Kaimanovich's extension of Oseledets Theorem to more general
semi-simple groups (\cite{K3}).

A metric geodesic space $(X,d)$ is called a $CAT(0)$ space if its
geodesic triangles are thinner than in the Euclidean space. Namely,
consider four points $A,B,C,D\in X$, $D$ lying on a length minimizing
geodesic going from $B$ to $C$. Draw four points $A',B',C',D'$
in the Euclidean plane with $AB=A'B',BD=B'D',DC=D'C',CA=C'A'.$ The
space is called $CAT(0)$ if, for any such configuration $AD\leq A'D'$.
Simply connected Riemannian spaces with nonpositive curvature, locally
finite trees and Euclidean buildings are proper $CAT(0)$ spaces.
If $X$ is a $CAT(0)$ space, then the horofunctions $h\in\partial X$
are called Busemann functions, and for any $h\in\partial X$, there
is a unique geodesic ray $\sigma_{h}(t),t\geq0$ such that $\sigma_{h}(0)=x_{0}$
and $\lim_{t\to\infty}\Phi_{\sigma_{h}(t)}=h.$ We have:

\begin{corollary}\label{ray} Assume moreover that $X$ is a $CAT(0)$
space and that $\alpha>0$. Then, for $\mathbb{P}$-almost every $\omega$,
as $n$ goes to $\infty$, \[
\lim_{n}\frac{1}{n}d\left(Z_{n}(\omega)x_{0},\sigma_{h_{\omega}}(\alpha n)\right)\;=\;0,\]
 where $h_{\omega}$ is given by Theorem \ref{main}. \end{corollary}
\begin{proof} Consider a geodesic triangle $A=Z_{n}(\omega)x_{0},B=x_{0},C_{t}=\sigma_{h_{\omega}}(t)$,
for $t$ very large, and choose $D=\sigma_{h_{\omega}}(n\alpha)$.
We want to estimate the distance $AD$. We have \begin{align*}
AB & =d(Z_{n}(\omega)x_{0},x_{0})=|Z_{n}(\omega)|=:n\alpha_{n}(\omega)\\
BC_{t} & =t,\; BD=n\alpha\textrm{ and }\\
C_{t}A & =t+\Phi_{\sigma_{h_{\omega}}(t)}(Z_{n}(\omega)x_{0})=:t-n\beta_{n}(\omega)+o_{n}(t).\end{align*}
 For almost every $\omega$, we have 
\begin{itemize}
\item $\lim_{n}\alpha_{n}(\omega)=\alpha$ by Theorem \ref{Kingman}, 
\item $\lim_{n}\beta_{n}(\omega)=\lim_{n}-\frac{1}{n}h_{\omega}(Z_{n}(\omega)x_{0})=\alpha$
by Theorem \ref{main} and 
\item for a fixed $n$, $\lim_{t\to\infty}o_{n}(t)=h_{\omega}(Z_{n}(\omega)x_{0})-h_{\sigma_{h_{\omega}}(t)}(Z_{n}(\omega)x_{0})=0.$ 
\end{itemize}
Construct the comparison figure $A'B'C'_{t}D'$, and let $t$ go to
$\infty$. The point $E'_{t}$ of $B'C'_{t}$ at the same distance
from $C'_{t}$ than $A'$ converges to the orthogonal projection $E'_{\infty}$
of $A'$ on $B'C'_{t}$ and satisfies $B'E'_{t}=n\beta_{n}-o_{n}(t)$.
Therefore, $B'E'_{\infty}=n\beta_{n}$. We have: \[
(AE'_{\infty})^{2}=n^{2}(\alpha_{n}^{2}-\beta_{n}^{2}),\;\;(D'E'_{\infty})^{2}=n^{2}(\beta_{n}-\alpha)^{2},\]
 and therefore, as $n\to\infty$: \[
\lim_{n}\frac{1}{n^{2}}(A'D')^{2}\;=\;\lim_{n}\left((\alpha_{n}^{2}-\beta_{n}^{2})+(\beta_{n}-\alpha)^{2}\right)=0.\]
 \end{proof}

\begin{corollary}\label{conv} With the same assumptions, we have,
for $\mathbb{P}$-almost every $\omega$, $Z_{n}(\omega)x_{0}$ converges
to $h_{\omega}$ in $\overline{X}$. \end{corollary} In particular,
when $\alpha>0$ and $X$ is proper $CAT(0)$, the direction $h_{\omega}$
given by Theorem \ref{main} is unique.

\begin{proof} In the above triangle, the geodesic $\sigma_{n}$ joining
$x_{0}$ to $Z_{n}(\omega)x_{0}$ converges to $\sigma_{h_{\omega}}$.
Therefore all the accumulation points of $Z_{n}(\omega)x_{0}$ belong
to the set seen from $x_{0}$ in the direction of $h_{\omega}$. By
the same proof, all the accumulation points of $Z_{n}(\omega)x_{0}$
belong to the set seen from $\sigma_{h_{\omega}}(K)$ in the direction
of $h_{\omega}$, for all $K$. As $K$ goes to infinity, the intersection
of those sets is reduced to the point $h_{\omega}$. \end{proof}

\

In the case when $G$ is a linear group, Corollary \ref{ray} is closely
related to the well known \begin{theorem}\label{OMET}{[}Oseledets
Multiplicative Ergodic Theorem, \cite{O}, 1968{]} Let $T$ be an
ergodic transformation of the Probability space $(\Omega,\mathcal{A},\mathbb{P})$,
and $A:\Omega\to GL(d,\mathbb{R})$ a measurable map such that $\int\max\{\ln||A||,\ln||A^{-1}||\}d\mathbb{P}<\infty$.
Then there exist 
\begin{itemize}
\item real numbers $\lambda_{1}\leq\lambda_{2}\leq\dots\leq\lambda_{k}$ 
\item integers $m_{i},i=1,\dots,k$ with $\sum_{i}m_{i}=d$, $\sum_{i}\lambda_{i}m_{i}=\int\ln|DetA|d\mathbb{P}.$ 
\item for $\mathbb{P}$-almost every $\omega$, a flag of subspaces of $\mathbb{R}^{d}$
\[
\{0\}=V_{k+1}(\omega)\subset V_{k}(\omega)\subset\dots\subset V_{1}(\omega)=\mathbb{R}^{d}\]
 
\end{itemize}
with, for all $i,1\leq i\leq k,$ $DimV_{i}=\sum_{j\geq i}m_{j}$
and a vector $v$ belongs to $V_{i}(\omega)\setminus V_{i+1}(\omega)$
if, and only if, as $n$ goes to $\infty$, \[
\lim\frac{1}{n}\ln||A(T^{n-1}\omega)A(T^{n-2}\omega)\dots A(\omega)v||=\lambda_{i}.\]
 \end{theorem}

Observe that, automatically, the $V_{i}$ depend measurably of $\omega$
and are invariant in the sense that $A(\omega)V_{i}(\omega)=V_{i}(T\omega)$.
The usual complete form of Oseledets Theorem follows by comparing
the results of Theorem \ref{OMET} for $(T,A)$ and for $(T^{-1},A^{-1}\circ T^{-1})$.
Fix $\omega\in\Omega$, and let $e_{i},i=1,\dots,d$ be an orthogonal
base of $\mathbb{R}^{d}$ such that $e_{\ell}\in V_{i}(\omega)$ as
soon as $\ell\leq\sum_{j\geq i}m_{j}$. Write $\mu_{1}\geq\dots\geq\mu_{d}$
for the exponents $\lambda_{j}$, each counted with multiplicity $m_{j}$,
and consider $A^{(n)}(\omega):=A(T^{n-1}\omega)A(T^{n-2}\omega)\dots A(\omega)$
in the base $(e_{i})$. To verify the statement of Theorem \ref{OMET},
it suffices to show that for all $\varepsilon>0$ and for $n$ large
enough, \[
\big|A_{i,j}^{(n)}(\omega)\big|\leq e^{n(\mu_{i}+\varepsilon)}\;\textrm{ and }\big|\ln|DetA^{(n)}(\omega)|-\sum_{j}\mu_{j}\big|\leq\varepsilon.\]

With the notations of Section 2, consider the action by isometries
of $GL(d,\mathbb{R})$ on the symmetric space $GL(d,\mathbb{R})/O(d,\mathbb{R})$
with origin $x_{0}=O(d,\mathbb{R})$ and distance $|g|=\sqrt{\sum_{j=1}^{d}(\ln\tau_{i})^{2}}$,
where $\tau_{i}$ are the eigenvalues of $gg^{t}$. It is a $CAT(0)$
geodesic proper space. Set $g(\omega)=A^{t}(\omega)$. The moment
hypothesis $\int|g|d\mathbb{P}<\infty$ is satisfied. We have $A^{(n)}(\omega)=(Z_{n}(\omega))^{t}.$
If $\alpha=0$, then the eigenvalues of $Z_{n}Z_{n}^{t}$ grow subexponentially
and \[
\lim_{n}\frac{1}{n}\ln||A^{(n)}(\omega)v||=\frac{1}{2}\lim_{n}\frac{1}{n}\ln(||Z_{n}^{t}v||^{2})=0.\]
 In this case $m_{1}=d,\lambda_{1}=0$ and Theorem \ref{OMET} holds.
We may assume $\alpha>0$, and apply Corollary \ref{ray}.

Geodesics starting from the origin are of the form $e^{tH}$, where
$H$ is a nonzero symmetric matrix. Therefore, for $\mathbb{P}$-almost
every $\omega$, there is a nonzero symmetric matrix $H(\omega)$
such that $\frac{1}{n}d(exp(nH(\omega)),Z_{n}(\omega))$ goes to $0$
as $n\to\infty$ (the constant $\alpha$ has been incorporated in
$H$). In other words, $\frac{1}{n}\ln$ of the norm, and of the norm
of the inverse, of the matrix $exp(-nH(\omega))(A^{(n)}(\omega))^{t}$
go to $0$ as $n\to\infty.$ We claim that this gives the conclusion
of Theorem \ref{OMET} with $\lambda_{i}$ the eigenvalues of $H(\omega)$,
$m_{i}$ their respective multiplicities and $V_{i}$ the sums of
the eigenspaces corresponding to eigenvalues smaller than $\lambda_{i}$.
Indeed, we write $exp(H(\omega))=K(\omega)^{t}\Delta K(\omega)$ for
$K$ an orthogonal matrix and $\Delta$ a diagonal matrix with diagonal
entries $e^{\mu_{i}}$, and $A^{(n)}(\omega)=L_{n}(\omega)\Delta_{n}(\omega)K_{n}(\omega)$
a Cartan decomposition of $A^{(n)}$ with $L_{n},K_{n}$ orthogonal,
$\Delta_{n}$ a diagonal matrix with nonincreasing diagonal entries
$exp(n\delta_{1}^{(n)}(\omega))\geq\dots\geq exp(n\delta_{d}^{(n)}(\omega))$.
The conclusion of Corollary \ref{ray} is therefore that, for $\mathbb{P}$-almost
every $\omega$, $\frac{1}{n}\ln$ of the norm, and of the norm of
the inverse, of the matrix $\Delta_{n}(\omega)K_{n}(\omega)K^{t}(\omega)exp(-n\Delta)$
go to $0$ as $n\to\infty.$

It follows that, for such an $\omega$, $\big|\ln|DetA^{(n)}(\omega)|-\sum_{j}\mu_{j}\big|$
goes to $0$ as $n$ goes to $\infty$. Furthermore, for $n$ large
enough, the entries $k_{i,j}^{(n)}(\omega)$ of the matrix $K_{n}(\omega)K^{t}(\omega)$
satisfy: \[
\big|k_{i,j}^{(n)}(\omega)\big|\;\leq\; e^{n(\mu_{j}-\delta_{i}^{(n)}+\varepsilon)}.\]
 We have \[
\|A^{(n)}(\omega)e_{j}\|\;=\;\|L_{n}(\omega)\Delta_{n}(\omega)K_{n}(\omega)e_{j}\|\;=\;\|\Delta_{n}(\omega)K_{n}(\omega)K^{-1}(\omega)f_{j}\|,\]
 where $f_{i}$ is the canonical base of $\mathbb{R}^{d}$. The components
of this vector are $e^{n\delta_{i}^{(n)}(\omega)}k_{i,j}^{(n)}(\omega)$.
Their absolute values are indeed smaller than $e^{n(\mu_{j}+\varepsilon)}$
for $n$ large enough.

\

\section{The case when $X$ is a Gromov hyperbolic space (in particular $\mathbb{R}$).}

Theorem \ref{main} is due to Kaimanovich using an idea of Delzant
when $X$ is a Gromov hyperbolic geodesic space even without the condition
that $X$ is a proper space \cite{K2}. As in the $CAT(0)$-case it
is there formulated as $Z_{n}$ lies on sublinear distance of a geodesic
ray. From Theorem \ref{main} one gets the following:

\begin{corollary}\label{hyperray} Assume moreover that $X$ is a
Gromov hyperbolic geodesic space and that $\alpha>0$. Then, for $\mathbb{P}$-almost
every $\omega,$ as $n$ goes to $\infty,$ there is a geodesic ray
$\sigma_{\omega}$ such that\[
\lim_{n}\frac{1}{n}d(Z_{n}(\omega)x_{0},\sigma_{h_{\omega}}(\alpha n))=0.\]

\end{corollary}

\begin{proof} Take $h_{\omega}$ given from Theorem \ref{main}.
It is known, see \cite[p. 428]{BH}, that for Gromov hyperbolic geodesic
spaces it holds that there is a geodesic ray $\sigma_{\omega}$ such
that $\sigma_{\omega}(x_{0})=0$ and\[
b_{\omega}(\cdot)=\lim_{t\rightarrow\infty}d(\cdot,\sigma_{\omega}(t))-t\]
 is a horofunction such that $\left\vert b_{\omega}(\cdot)-h_{\omega}(\cdot)\right\vert \leq C$
for some constant $C$. This $b_{\omega}$ therefore clearly satisfies
the conclusion of Theorem \ref{main}.

Now we use the notation and set-up in the proof of Corollary \ref{ray}.
Consider the triangle $ABC_{t}$. By $\delta$-hyperbolicity $D$
must lie at most $\delta$ away from either $AB$ or $AC_{t}$. Call
the closest point $X$. By the triangle inequality we must have that
\[
\alpha n-\delta\leq XB\leq\alpha n+\delta\text{.}\]
 If $X$ lie on $AB$, then it is clear that $XA=o(n)$ and hence
$AD=o(n)$. If $X$ lie on $AC_{t}$, then\[
t-\alpha n-\delta\leq XC_{t}\leq t-\alpha n+\delta.\]
 In view of that $b_{\omega}(Z_{n}(\omega))\approx-\alpha n$ we again
reach the conclusion that $X,$ and hence also $D,$ lie on sublinear
distance from $A$. \end{proof}

\begin{corollary}\label{hyperconv} With the same assumptions, we
have that for $\mathbb{P}$-almost every $\omega$, $Z_{n}(\omega)x_{0}$
converges to the point $[\sigma_{\omega}]$ in the hyperbolic boundary
$\partial_{hyp}X.$ \end{corollary}

\begin{proof} Clearly, the Gromov product $(Z_{n}(\omega),\sigma_{\omega}(\alpha n))\rightarrow\infty$
as $n\rightarrow\infty$ in view of the previous corollary. \end{proof}

\ 

In the case when $G=\mathbb{R}$ and $X=(\mathbb{R},\left\vert \cdot\right\vert ),$
Corollaries \ref{hyperray} and \ref{hyperconv} yield Theorem \ref{Birkhoff}.
Indeed, in this case the drift is:\[
\alpha=\left\vert \int_{\Omega}fd\mathbb{P}\right\vert \]
 and $\partial\mathbb{R=\{}h_{+}=\Phi_{+\infty}(z)=-z,$ $h_{-}=\Phi_{-\infty}(z)=z\mathbb{\}}.$
It follows from Corollary \ref{hyperconv} that the index of $h_{\omega}$
is $T$ invariant and is therefore almost everywhere constant. In
other words, the existence of the $h_{\omega}$ with the required
property amounts to the choice of the right sign:\[
h_{\omega}=\Phi_{sign\left\{ \int_{\Omega}f(\omega)d\mathbb{P}(\omega)\right\} \infty}.\]
 Then, Corollary \ref{hyperray} say exactly that if the function
$f$ is integrable, for $\mathbb{P}$-almost every $\omega$\[
\frac{1}{n}S_{n}(\omega)\rightarrow\int_{\Omega}f(\omega)d\mathbb{P}(\omega).\]

The above observation is not a new proof of Theorem \ref{Birkhoff},
because Theorem \ref{Birkhoff} is used in the proof of Theorem \ref{main}
(see section 5). We only want to illustrate the meaning of the metric
boundary on the simplest example. Nevertheless, it turns out that
modifying the translation invariant metric on $X=\mathbb{R}$ might
have interesting consequences. The following discussion comes from
\cite{KMo}, which in turn was inspired by \cite{LL}.

Let $D:\mathbb{R}_{\geq0}\rightarrow\mathbb{R}_{\geq0}$ be an increasing
function, $D(t)\rightarrow\infty$ such that $D(0)=0$ and $D(t)/t\rightarrow0$
monotonically. From the inequality\[
\frac{1}{t+s}D(t+s)\leq\frac{1}{t}D(t)\]
 we get the following subadditivity property\[
D(t+s)\leq D(t)+\frac{s}{t}D(t)=D(t)+\frac{D(t)/t}{D(s)/s}D(s)\leq D(t)+D(s).\]
 From all these properties of $D$, it follows that $(\mathbb{R},D(\left\vert \cdot\right\vert ))$
is a proper metric space, and clearly invariant under translations.

Now we determine $\partial\mathbb{R}$ with respect to this metric.
Wlog we may assume that $x_{n}\rightarrow\infty$. We claim that for
any $z$\[
h(z)=\lim_{n\rightarrow\infty}D(x_{n}-z)-D(x_{n})=0.\]
 Assume not. Then for some $s>0$ and an infinite sequence of $t\rightarrow\infty$
that $D(t+s)-D(t)>c>0$ (wlog). For such $s$ and $t$ with $t$ large
so that $D(t)/t<c/s,$ we have\[
\frac{D(t+s)}{t+s}\geq\frac{D(t)+c}{t+s}\geq\frac{D(t)+\frac{D(t)}{t}s}{t+s}=\frac{D(t)}{t}\]
 but this contradicts that $D(t)/t$ is strictly decreasing. Hence
$\partial\mathbb{R}=\{h\equiv0\}.$

Applying Theorem \ref{main} in this setting yields a result already
obtained by Aaronson with a different argument.

\begin{theorem}\label{Aaronson} {[}Aaronson \cite{A}{]} Let $f:\Omega\rightarrow\mathbb{R}$
such that $\int_{\Omega}D(\left\vert f\right\vert )d\mu<\infty.$Then,
for $\mathbb{P}$-almost every $\omega$, \[
\lim_{n\rightarrow\infty}\frac{1}{n}D\left(\left\vert S_{n}(\omega)\right\vert \right)=0.\]

\end{theorem}

\begin{proof} It was noted above that $\partial(\mathbb{R},D(\left\vert \cdot\right\vert ))$
only consisted of $h=0.$ The conclusion then follows from Theorem
\ref{main} since $h=0$ forces $\alpha=0.$ \end{proof}

\

One can relax the conditions on $D$: for one thing, one can remove
having $D(0)=0$. More interestingly, the condition that $D(t)/t$
decreases to $0$ can be weakened in the following way.

\begin{corollary}\label{d} {[}Aaronson-Weiss \cite{A}{]} Let $d(t)$
be an increasing positive function, $d(t)\rightarrow\infty$, such
that $d(t)=o(t)$, $d(t+s)\leq d(t)+d(s)$ and $\int_{\Omega}d(\left\vert f\right\vert )d\mu<\infty$
for some function $f:\Omega\rightarrow\mathbb{R}$. Then, for $\mathbb{P}$-almost
every $\omega$, \[
\lim_{n\rightarrow\infty}\frac{1}{n}d\left(\left\vert S_{n}(\omega)\right\vert \right)=0.\]
 \end{corollary}

\begin{proof} Define\[
D(t)=\sup\{d(ut)/u:u\geq1\}.\]
 Note that this satisfies all the assumptions made on $D$ in Theorem
\ref{Aaronson}. Moreover\[
d(t)\leq D(t)\leq2d(t),\]
 since if $D(t)=d(tu)/u,$ set $n=[u]+1$ and then $D(u)\leq d(nt)/u\leq nd(t)/u\leq2d(t).$
See \cite{A}, page 66, for more details. This shows that Theorem
\ref{Aaronson} actually holds for $d$ in place of $D$. \end{proof}

In particular, Corollary \ref{d} applies to any metric $d(.,.)$
on $\mathbb{R}$ where balls grow superlinearly (where $d(t):=d(0,t)$).
From this one obtains as a special case classical results like the
one of Marcinkiewicz-Zygmund \cite{MZ} and Sawyer (\cite{S}): \begin{corollary}
Let $0<p<1.$ If $f\in L^{p},$ then for $\mathbb{P}$-almost every
$\omega$ \[
\lim_{n\rightarrow\infty}\frac{1}{n^{1/p}}S_{n}=0.\]

\end{corollary}

Such moment conditions arise naturally in probability theory. These
results are known to be best possible in certain ways (e.g. \cite{S}
and \cite{A}). For the iid case the converse also holds (\cite{MZ}).
Another example

\begin{corollary} If $f$ is log-integrable, then for $\mathbb{P}$-almost
every $\omega$ \[
\lim_{n\rightarrow\infty}\left\vert S_{n}\right\vert ^{1/n}=1.\]

\end{corollary}

One can modify the metric on any metric space $X$ in the same way
replacing $d(x,y)$ with $D(d(x,y))$, where $D(t)$ satisfies the
assumptions for Theorem \ref{Aaronson} or, more generally, the assumptions
for Corollary \ref{d}. By estimating a subadditive by an additive
cocycle in the obvious way,\[
a(n,\omega)\leq\sum_{k=0}^{n-1}a(1,T^{k}\omega),\]
 Theorem \ref{Aaronson} implies that\[
\frac{1}{n}D(d(Z_{n}x_{0},x_{0}))\rightarrow0\text{ a.e.}\]
 under the condition that $D(d(g(\omega)x_{0},x_{0}))$ is integrable.

\section{Proof of Theorem \ref{main}.}

We begin by a few observations: firstly, we can extend by continuity
the action of $G$ to $\overline{X}$, and write, for $h\in\overline{X},g\in G$:
\[
g.h(z)\;=\; h(g^{-1}z)-h(g^{-1}x_{0}).\]
 Define now the skew product action on $\overline{\Omega}:=\Omega\times\overline{X}$
by: \[
\overline{T}(\omega,h)\;=(T\omega,g(\omega)^{-1}.h).\]
 Observe that $\overline{T}^{n}(\omega,h)\;=(T^{n}\omega,(Z_{n}(\omega))^{-1}.h).$
Define the Furstenberg cocycle $\overline{F}(\omega,h)$ by $\overline{F}(\omega,h):=-h(g(\omega)x_{0}).$
We have: \begin{equation}
\overline{F}_{n}(\omega,h)\;:=\;\sum_{i=0}^{n-1}\overline{F}(\overline{T}^{i}(\omega,h))\;=\;-h(Z_{n}(\omega)x_{0}).\label{cocycle}\end{equation}
 Relation (\ref{cocycle}) is proven by induction on $n$. We have
$\overline{F}_{1}(\omega,h):=-h(g(\omega)x_{0})=-h(Z_{1}(\omega)x_{0})$
and \begin{align*}
\overline{F}_{n}(\omega,h)\; & =\;\overline{F}_{n-1}(\omega,h)+\overline{F}(\overline{T}^{n-1}(\omega,h))\\
 & =\;-h(Z_{n-1}(\omega)x_{0})-(Z_{n-1}(\omega))^{-1}.h(g(T^{n-1}\omega)x_{0})\\
 & =\;-h(Z_{n-1}(\omega)x_{0})-h(Z_{n-1}(\omega)g(T^{n-1}(\omega))x_{0})+h(Z_{n-1}(\omega)x_{0})\\
 & =\;-h(Z_{n}(\omega)x_{0}).\end{align*}

In particular, for any $\overline{T}$ invariant measure $m$ on $\overline{\Omega}$
such that the projection on $\Omega$ is $\mathbb{P}$, we have $\int\overline{F}dm\leq\alpha$
because: \[
\int\overline{F}(\omega,h)dm(\omega,h)\;=\;\frac{1}{n}\int-h(Z_{n}(\omega)x_{0})dm(\omega,h)\;\leq\frac{1}{n}\int|Z_{n}(\omega)|d\mathbb{P}(\omega).\]

There is nothing to prove if $\alpha=0$. To prove Theorem \ref{main}
in the case $\alpha>0$, it suffices to construct a $\overline{T}$
invariant measure $m$ on $\overline{\Omega}$ such that the projection
on $\Omega$ is $\mathbb{P}$ and such that $\int\overline{F}(\omega,h)dm(\omega,h)=\alpha.$
Indeed, since $\alpha$ is the largest possible value of $\int\overline{F}$,
we still have the same equality for almost every ergodic component
of $m$. By the Ergodic Theorem \ref{Birkhoff}, the set $A$ of $(\omega,h)$
such that $-\frac{1}{n}h(Z_{n}(\omega)x_{0})=\frac{1}{n}\overline{F}_{n}(\omega,h)\rightarrow\alpha$
as $n$ goes to $\infty$ has full measure. Moreover, observe that
if $h$ is not a point in $\partial X$, $-\frac{1}{n}h_{\gamma}(Z_{n}(\omega)x_{0})$
converges to $-\alpha$. Since $\alpha>0$, this shows that $A\subset\partial X$.
We get the conclusion of Theorem \ref{main} by choosing for $\omega\mapsto h_{\omega}$
a measurable section of the set $A$.

\
 We finally construct a measure $m$ with those properties. We define
a measure $\mu_{n}$ on $\overline{\Omega}$; for any measurable function
$\Xi$ on $\overline{\Omega}$ such that \[
\int\sup_{h\in\overline{X}}\left|\Xi(\omega,h)\right|d\mathbb{P}(\omega)<\infty,\]
 we set: \[
\int_{\Omega\times\overline{X}}\Xi(\omega,h)d\mu_{n}(\omega,h)\;=\;\int_{\Omega}\Xi(\omega,\Phi_{Z_{n}(\omega)x_{0}})d\mathbb{P}(\omega).\]
 The set of measures $m$ on $\overline{\Omega}$ such that the projection
on $\Omega$ is $\mathbb{P}$ is a convex compact subset of $L^{\infty}(\Omega,\mathcal{P}(\overline{X}))=(L^{1}(\Omega,C(\overline{X})))^{\ast}$
for the weak{*} topology. The mapping $m\mapsto(\overline{T})_{\ast}m$
is affine and continuous. We can take for $m$ any weak{*} limit point
of the sequence: \[
\eta_{n}\;=\;\frac{1}{n}\sum_{i=0}^{n-1}(\overline{T}^{i})_{\ast}\mu_{n}.\]
 The measure $m$ is $\overline{T}$ invariant and, since \[
\|\overline{F}\|_{L^{1}(\Omega,C(\overline{X}))}=\int\sup_{h}\big|\overline{F}(\omega,h)\big|d\mathbb{P}(\omega)=\int\sup_{h}|h(g(\omega)x_{0})|d\mathbb{P}(\omega)<+\infty,\]
 we may write, using relation (\ref{cocycle}) and the formula (\ref{alpha}):
\begin{align*}
\int\overline{F}dm\; & =\;\lim_{k\to\infty}\frac{1}{n_{k}}\int\sum_{i=0}^{n_{k}-1}(\overline{F}\circ\overline{T}^{i})d\mu_{n_{k}}\\
 & =\;\lim_{k\to\infty}\frac{1}{n_{k}}\int\overline{F}_{n_{k}}(\omega,\Phi_{Z_{n_{k}}(\omega)x_{0}})d\mathbb{P}(\omega)\\
 & =\;\lim_{k\to\infty}\frac{1}{n_{k}}\int(-\Phi_{Z_{n_{k}}(\omega)x_{0}}(Z_{n_{k}}(\omega)x_{0}))d\mathbb{P}(\omega)\\
 & =\;\lim_{k\to\infty}\frac{1}{n_{k}}\int|Z_{n_{k}}(\omega)|d\mathbb{P}(\omega)\;=\alpha.\end{align*}
 By the above discussion this achieves the proof of Theorem \ref{main}.

\
 Observe that by putting together the discussions in sections 5 and
3, we obtain a proof of Oseledets Theorem \ref{OMET}. As proofs of
Theorem \ref{OMET} go, this one is in some sense rather close to
the original one (\cite{O}), with the somewhat simplifying use of
the geometric ideas from \cite{K3} and invariant measures as in \cite{W}.

\section{ Random Walks.}

In this section we consider a probability $\nu$ on a group $G$ and
apply the preceeding analysis to the random walk $Z_{n}=g_{0}g_{1}\dots g_{n-1}$,
where the $g_{i}$ are independent with distribution $\nu$. We assume: 
\begin{itemize}
\item there is a proper left invariant metric $d$ on $G$ which generates
the topology of $G$ (when $G$ is second countable locally compact,
such a metric always exists, see \cite{St}), 
\item $\int d(e,g)d\nu(g)<+\infty$ (we say that $\nu$ has a \textit{first
moment}) and 
\item the closed subgroup generated by the support of $\nu$ is the whole
$G$ (we say that $\nu$ is \textit{non-degenerate}). 
\end{itemize}
Then, there is a number $\ell(\nu)\geq0$ such that, for almost every
sequence $\{g_{i}\}$, $\lim_{n}\frac{1}{n}d(e,Z_{n})=\ell(\nu).$
In the case when the group $G$ is the group $SL_{2}(\mathbb{R})$
acting on the hyperbolic plane, $\ell(\nu)$ is twice the Lyapunov
exponent of the independent product of matrices. In that case it is
given by a formula involving the stationary measure on the circle,
the Furstenberg-Khasminskii formula (\cite{F1}; this appellation
seems to be standard, cf. \cite{Ar}). Seeing again the circle as
the geometric boundary of the hyperbolic plane, we extend this formula
to our general context: \begin{theorem}\label{FK}{[}Furstenberg-Khasminskii
formula for the linear drift, \cite{KL2}{]}. Let $(G,\nu)$ verify
all the above assumptions, and let $\overline{G}$ be the metric compactification
of $(G,d)$. Then there exists a measure $\mu$ on $\overline{G}$
with the following properties: 
\begin{itemize}
\item $\mu$ is stationary for the action of $G$, i.e. $\mu$ satisfies
$\mu=\int(g_{\ast}\mu)d\nu(g)$ and 
\item $\ell(\nu)\;=\;\int h(g^{-1})d\mu(h)d\nu(g).$ 
\end{itemize}
Moreover, if $\ell(\nu)>0$, then $\mu$ is supported on $\partial G$.
\end{theorem}

\begin{proof} In the proof of Theorem \ref{main}, we constructed
a measure $m$ on $\Omega\times\overline{G}$. The measure $\mu$
can be seen as the projection on $\overline{G}$ of $m$, but it turns
out that the measure $\mu$ can be directly constructed. Let $(\Omega^{+},\mathcal{A}^{+},\mathbb{P})$
be the space of sequences $\{g_{0},g_{1},\dots\}$ with product topology,
$\sigma$-algebra and measure $\mathbb{P}=\nu^{\otimes\mathbb{N}}$.
For $n\geq0$, let $\nu_{n}$ be the distribution of $Z_{n}(\omega)$
in $\overline{G}$. In other words, define, for any continuous function
$f$ on $\overline{G}$: \[
\int fd\nu_{n}\;=\;\int f(g_{0}g_{1}\cdots g_{n-1})d\nu(g_{0})d\nu(g_{1})\cdots d\nu(g_{n-1}),\quad\nu_{0}=\delta_{e}.\]
 We claim that any weak{*} limit ${\mu}$ of the measures $\frac{1}{n}\sum_{i=0}^{n-1}\nu_{i}$
satisfies the conclusions of Theorem \ref{FK}. Clearly, the measure
$\mu$ is stationary: for any continuous function $f$ on $\overline{G}$,
we have \begin{eqnarray*}
 & \int f(g.h)d\mu(h)d\nu(g)\\
 & =\;\lim_{k\to\infty}\frac{1}{n_{k}}\sum_{i=0}^{n_{k}-1}\int f(gg_{0}g_{1}\cdots g_{i-1})d\nu(g_{0})d\nu(g_{1})\cdots d\nu(g_{i-1})d\nu(g)\\
 & \;=\;\lim_{k\to\infty}\frac{1}{n_{k}}\sum_{i=0}^{n_{k}-1}\int fd\nu_{i+1}\\
 & =\;\int fd\mu+\lim_{k\to\infty}\frac{1}{n_{k}}[\int fd\nu_{n_{k}}-f(e)]\;=\;\int fd\mu.\end{eqnarray*}
 In the same way, we get: \begin{eqnarray*}
 & \int h(g^{-1})d\mu(h)d\nu(g)\;\\
= & \;\lim_{k\to\infty}\frac{1}{n_{k}}\sum_{i=0}^{n_{k}-1}\int[d(Z_{i}(\omega),g^{-1})-d(Z_{i}(\omega),e)]d\mathbb{P}(\omega)d\nu(g)\\
= & \;\lim_{k\to\infty}\frac{1}{n_{k}}\int d(Z_{n_{k}},e)d\mathbb{P}(\omega)\;=\;\ell(\nu).\end{eqnarray*}

This shows that the measure $\mu$ has the desired properties. Moreover,
the measure $\mathbb{P}\times\mu$ on the space $\Omega^{+}\times\overline{G}$
is $\overline{T}$-invariant. There is a unique $\overline{T}$-invariant
measure $m$ on $\Omega\times\overline{G}$ that extends $\mathbb{P}\times\mu$.
The measure $m$ satisfies all the properties we needed in the proof
of Theorem \ref{main}. In particular, if $\ell(\nu)$ is positive,
\[
\mu(\partial G)\;=\;(\mathbb{P}\times\mu)(\Omega^{+}\times\partial G)\;=\; m(\Omega\times\partial G)\;=\;1.\]
 \end{proof}

\
 A bounded measurable $f:G\rightarrow\mathbb{R}$ is $\nu$-\emph{harmonic}
if \[
f(g)=\int_{G}f(gh)d\nu(h)\]
 for any $g\in G$. Constant functions are obviously $\nu$-harmonic.
If $f$ is a bounded harmonic function, then $f(Z_{n})$ is a bounded
martingale and therefore converges almost surely. We say that $(G,\nu)$
satisfies the \textit{Liouville property} (or $(G,\nu)$ is Liouville)
if the constant functions are the only bounded $\nu$-harmonic functions.

\begin{corollary} \label{thmkl}\cite{KL2} Let $G$ be a locally
compact group with a left invariant proper metric and $\nu$ be a
nondegenerate probability measure on $G$ with first moment. Then,
if $(G,\nu)$ is Liouville, there is a $1$-Lipschitz homomorphism
$T:G\rightarrow\mathbb{R}$ such that for almost every trajectory
$Z_{n}$ of the corresponding random walk, we have: \[
\lim_{n\rightarrow\infty}\frac{1}{n}T(Z_{n})\;=\;\int_{G}T(g)d\nu(g)\;=\; l(\nu).\]
 \end{corollary}

\begin{proof} The key observation is that if $(G,\nu)$ is Liouville
and $G$ acts continuously on a compact space $Y$, then every stationary
measure $\mu$ is invariant. Indeed, for $f\in C(Y,\mathbb{R})$,
the function $\varphi(g):=\int fd(g_{\ast}\mu)$ is harmonic and bounded,
therefore constant. In particular, the measure $\mu$ from Theorem
\ref{FK} is invariant, and if we set \[
T(g):=\int h(g^{-1})\mu(dh),\]
 The mapping $T$ is Lipschitz continuous and is a group homomorphism
because we have: \begin{align*}
T(g'g)\; & =\;\int h(g^{-1}g'^{-1})\mu(dh)\\
 & =\;\int(g'.h)(g^{-1})\mu(dh)+\int h(g'^{-1})\mu(dh)\\
 & =\;\int h(g^{-1})(g'_{\ast}\mu)(dh)+T(g')\\
 & =\; T(g)+T(g'),\end{align*}
 where we used the invariance of $\mu$ at the last line. Finally,
by the Furstenberg Khasminskii formula, we have: \[
\ell(\nu)=\int T(g)d\mu(g).\]
 \end{proof}

\

A measure $\nu$ on $G$ is called \emph{symmetric} if it is invariant
under the mapping $g\mapsto g^{-1}$. A measure is \emph{centered}
if every homomorphism of $G$ into $\mathbb{R}$ is centered, meaning
that the $\nu$-weighted mean value of the image is 0. Every symmetric
measure with first moment $\nu$ is centered, since for any homomorphism
$T:G\rightarrow\mathbb{R}$, the mean value, which is \[
\int_{G}T(g)d\nu(g)=\int_{G}T(g^{-1})d\nu(g)=-\int_{G}T(g)d\nu(g),\]
 must hence equal 0. By simple contraposition from Corollary \ref{thmkl},
we get:\begin{corollary}\label{Varo}\cite{KL2} Let $G$ be a locally
compact group with a left invariant proper metric and $\nu$ be a
nondegenerate centered probability measure on $G$ with first moment.
Then, if $l(\nu)>0$, there exist nonconstant bounded $\nu$-harmonic
functions. \end{corollary}

Corollary \ref{Varo} was known in particular for $\nu$ with finite
support (\cite{V}, \cite{M}) or in the continuous case, for $\nu$
with compact support and density (\cite{Al}).

\

One case when all probability measures on $G$ are centered is when
there is no group homomorphism from $G$ to $\mathbb{R}$. We can
apply Corollary \ref{Varo} to a countable finitely generated group.
Let $S$ be a finite symmetric generator for $G$, and endow $G$
with the left invariant metric $d(x,y)=|y^{-1}x|$ where $|z|$ is
the shortest length of a $S$-word representing $z$. We say that
$G$ has subexponential growth if $\lim_{n}\frac{1}{n}\ln A_{n}=0$,
where $A_{n}$ is the number of elements $z$ of $G$ with $|z|\leq n.$
Such a group has automatically the Liouville property (\cite{Av}).
This yields:

\begin{corollary}\label{what?} \cite{KL2} Let $G$ be a finitely
generated group with subexponential growth and $H^{1}(G,\mathbb{R})=0$.
Then for any nondegenerate $\nu$ on $G$ with first moment, we have
$\ell(\nu)=0$. \end{corollary}

Observe that conversely, if there exists a nontrivial group homomorphism
$T$ from a finitely generated group into $\mathbb{R}$, then there
exists a nondegenerate probability $\nu$ on $G$, with first moment
and $\ell(\nu)>0$. Indeed, there exists $M$ such that $[-M,M]$
contains all the images of the elements of the generating set $S$.
We can choose $\nu$ carried by all the elements of $S$ with images
in $[0,M]$. Since $T$ is nontrivial, $\int T(g)d\nu(g)>0$. The
measure $\nu$ is nondegenerate, has finite support and $\ell(\nu)>0$
since for all $g\in G,|T(g)|\leq M|g|$.

\

\section{Riemannian covers.}

In this section we consider a complete connected Riemannian manifold
$(M,g)$ with bounded sectional curvatures. In particular, if $d_{M}$
is the Riemannian distance on $M$, $(M,d)$ is a proper space. Associated
to the metric is the Laplace-Beltrami operator $\Delta.$ A function
$f$ is \emph{harmonic} if $\Delta f=0$. We say that $M$ is \emph{Liouville}
if all bounded and harmonic functions are constant.

Associated to $\Delta$ is a diffusion process $B_{t}$ called \emph{Brownian
motion}. Since the curvature is bounded and $M$ is complete, the
Brownian motion is defined for all time. For all $x\in M$, there
is a probability $\mathbb{P}_{x}$ on $C(\mathbb{R}_{+},M)$ such
that the process $B_{t}$ given by the $t$ coordinate is a Markov
process with generator $\Delta$ and $B_{0}=x$. We can define \[
\ell_{g}:=\limsup_{t\rightarrow\infty}\frac{1}{t}d_{M}(x_{0},B_{t}),\]
 for any $x_{0}\in M$. \begin{theorem} \label{BM} \cite{KL3} Assume
that $(M,g)$ is a regular covering of a Riemannian manifold which
has finite Riemannian volume and bounded sectional curvatures. Then
$M$ is Liouville if, and only if, \[
\lim_{t\rightarrow\infty}\frac{1}{t}d(x_{0},B_{t})=0\textrm{ a.s.}.\]

\end{theorem}

The \textquotedbl{}if\textquotedbl{} part was proved by Kaimanovich,
see \cite{K1}, and the converse is clear if the Brownian motion is
recurrent on $M$. The proof of the new implication in Theorem \ref{BM}
in the transient case uses the Furstenberg-Lyons-Sullivan discretization
procedure. Let $\Gamma$ be the covering group of isometries of $M$.
This discretization consists in the construction of a probability
measure $\nu$ on $\Gamma$, with the following properties: 
\begin{itemize}
\item The restriction $f(\gamma):=F(\gamma x_{0})$ is a one-to-one correspondence
between bounded harmonic functions on $M$ and bounded functions on
$\Gamma$ which satisfy \[
f(\gamma)=\sum_{g\in\Gamma}f(\gamma g)\nu(g)\]
. 
\item If $\gamma_{1},\dots,\gamma_{n}$ are chosen independent and with
distribution $\nu$, then $\lim_{n\rightarrow\infty}\frac{1}{n}d_{M}(x_{0},\gamma_{1}\dots\gamma_{n}x_{0})$
exists. It vanishes a.e. if, and only if, $\lim_{t\rightarrow\infty}\frac{1}{t}d_{M}(x_{0},B_{t})=0$
a.s.. 
\item In the case when the Brownian motion is transient, one can choose
$\nu$ symmetric, i.e. such that for all $\gamma$ in $\Gamma$, $\nu(\gamma^{-1})=\nu(\gamma)$. 
\end{itemize}
The first property goes back to Furstenberg (\cite{F2}) and has been
systematically developed by Lyons and Sullivan (\cite{LS}) and Kaimanovich
(\cite{K5}). The second one was observed in certain situations by
Guivarc'h (\cite{G}) and Ballmann (\cite{Ba}). Babillot observed
that the modified construction of \cite{BL2} has the symmetry property.
Given the above, proving Theorem \ref{BM} mostly reduces to Corollary
\ref{Varo}, if we can show that hypotheses of Corollary \ref{Varo}
are satisfied. We endow $\Gamma$ with the metric defined by the metric
of $M$ on the orbit $\Gamma x_{0}$. This defines a left invariant
and proper metric on $\Gamma$: bounded sets are finite, because they
correspond to pieces of the orbit situated in a ball of finite volume.
The measure $\nu$ is nondegenerate because its support is the whole
$\Gamma$. It is shown in \cite{KL3} that the measure $\nu$ has
a first moment. The proof uses the details of the construction, but
the idea is that the distribution of $\nu$ is given by choosing some
random time and looking at the point $\gamma x_{0}$ close to the
trajectory of the Brownian motion at that time. Since the curvature
is bounded from below, if the expectation of the time is finite, the
expectation of the distance of the Brownian point at that time is
finite as well. It also follows that the rates of escape of the Brownian
motion and of the Random walk are proportional. Therefore, if the
manifold $(M,g)$ is Liouville, then the Random walk $(G,\nu)$ is
Liouville. By Corollary \ref{Varo}, $\lim_{n\rightarrow\infty}\frac{1}{n}d_{M}(x_{0},\gamma_{1}\dots\gamma_{n}x_{0})=0$
and therefore, $\lim_{t\rightarrow\infty}\frac{1}{t}d_{M}(x_{0},B_{t})=0$
a.s..

\
 There are many results about the Liouville property for Riemannian
covers of a compact manifold. Theorem \ref{BM} implies that the corresponding
statements hold for the rate of escape of the Brownian motion. Guivarc'h
(\cite{G}) showed that if the group $\Gamma$ is not amenable, then
$(M,g)$ is not Liouville, whereas when $\Gamma$ is polycyclic, $(M,g)$
is Liouville (Kaimanovich \cite{K2}). Lyons and Sullivan (\cite{LS},
see also \cite{E} for a simply connected example) have examples of
amenable covers without the Liouville property.

\
 \textbf{Acknowledgements.} This survey grew out from lectures given
by both authors at the IIIème Cycle Romand de Mathématiques in Les
Diablerets in March 2008 and from conversations there. We thank Tatiana
Smirnova-Nagnibeda and Slava Grigorchuk for their invitation to this
friendly and stimulating meeting.

\end{document}